\documentclass[12pt]{amsart}
\usepackage{amsmath,amssymb}
\title[Examples of groups which are not weakly amenable]
{\large E\lowercase{xamples of groups which are not weakly amenable}}
\author[N. Ozawa]{N\lowercase{arutaka} OZAWA}
\address{Department of Mathematical Sciences,
University of Tokyo, \mbox{153-8914}}
\email{narutaka@ms.u-tokyo.ac.jp}
\thanks{Partially supported by JSPS and Sumitomo Foundation}

\date{December 01, 2010}
\newtheorem{thm}{Theorem}
\newtheorem{lem}[thm]{Lemma}
\newtheorem{prop}[thm]{Proposition}
\newtheorem{cor}[thm]{Corollary}
\newtheorem{thmA}{Theorem}

\theoremstyle{definition}
\newtheorem*{defn}{Definition}

\newcommand{\IR}{{\mathbb R}}

\newcommand{\IN}{{\mathbb N}}
\newcommand{\IZ}{{\mathbb Z}}

\newcommand{\IB}{{\mathbb B}}
\newcommand{\acts}{\curvearrowright}
\newcommand{\G}{\Gamma}

\newcommand{\p}{\varphi}
\newcommand{\cA}{{\mathcal A}}
\newcommand{\cG}{{\mathcal G}}
\newcommand{\cH}{{\mathcal H}}
\newcommand{\cL}{{\mathcal L}}
\newcommand{\cN}{{\mathcal N}}
\newcommand{\cU}{{\mathcal U}}
\DeclareMathOperator*{\Lim}{Lim}
\DeclareMathOperator{\Ad}{Ad}
\newcommand{\id}{\mathrm{id}}
\newcommand{\cb}{\mathrm{cb}}
\newcommand{\SL}{\mathrm{SL}}
\newcommand{\Sp}{\mathrm{Sp}}
\newcommand{\vt}{\mathbin{\bar{\otimes}}}
\newcommand{\ip}[1]{\mathopen{\langle}#1\mathclose{\rangle}}
\begin{document}
\begin{abstract}
We prove that weak amenability of a locally compact group imposes
a strong condition on its amenable closed normal subgroups.
This extends non weak amenability results of Haagerup (1988) and Ozawa--Popa (2010).
A von Neumann algebra analogue is also obtained.
\end{abstract}
\maketitle
\section{Introduction}
Let $G$ be a group, which is always assumed to be a locally compact topological group.
The group $G$ is said to be \emph{weakly amenable} if
the Fourier algebra $\cA G$ of $G$ has an approximate identity $(\p_n)$
which is uniformly bounded as Herz--Schur multipliers.
(If one requires $(\p_n)$ to be bounded as elements in $\cA G$,
it becomes one of the equivalent definitions of amenability.)
See Section~\ref{sec:schur} for the precise definition. Weak amenability is strictly
weaker than amenability and passes to closed subgroups.
It is proved by De Canni\`ere--Haagerup, Cowling and Cowling--Haagerup (\cite{dch,cowling,ch})
that real simple Lie groups of real rank one are weakly amenable (see also \cite{hypcbap}),
and by Haagerup (\cite{haagerup:u}) that real simple Lie groups of real rank at least two
are not weakly amenable. For the latter fact, Haagerup proves that $\SL(2,\IR)\ltimes\IR^2$
is not weakly amenable. (See also \cite{dorofaeff}.)
More recently, it is proved by Ozawa--Popa (\cite{cartan}) that
the wreath product $\Lambda\wr\G$ of a non-trivial group $\Lambda$
by a non-amenable discrete group $\G$ is not ``weakly amenable with constant $1$.''
In this paper, we generalize these non weak amenability results as follows.

\begin{thmA}\label{thmA}
Let $G$ be an weakly amenable group and $N$ be an amenable closed normal
subgroup of $G$. Then, there is a $G\ltimes N$-invariant state on
$L^\infty(N)$, where the semidirect product $G\ltimes N$
acts on $N$ by $(g,a)\cdot x= gaxg^{-1}$.
\end{thmA}

In particular, the wreath product by a non-amenable group is never weakly amenable.
The theorem also gives a new proof of Haagerup's result that
$\SL(2,\IZ)\ltimes\IZ^2$ is not weakly amenable, without appealing to
the lattice embedding into $\SL(2,\IR)\ltimes\IR^2$.
We note for the sake of completeness that there is an even weaker variant of
weak amenability, called the \emph{approximation property} (\cite{hk}),
and $\SL(2,\IR)\ltimes\IR^2$ has the approximation property,
while $\SL(n\geq 3,\IR)$ does not (\cite{ls}).

As Theorem 3.5 in \cite{cartan}, there is an analogous result for von Neumann algebras.
We refer to Section~3 in \cite{cartan} and Section~\ref{sec:b} of this paper for the terminology
used in the following theorem.
\begin{thmA}\label{thmB}
Let $M$ be a finite von Neumann algebra with
the weak$^*$ completely bounded approximation property.
Then, every amenable von Neumann subalgebra $P$ is weakly compact in $M$.
\end{thmA}

It follows that a type $\mathrm{II}_1$ factor having the weak$^*$ completely bounded
approximation property and property $\mathrm{(T)}$ (e.g., the group
von Neumann algebra of a torsion-free lattice in $\Sp(1,n)$)
is not isomorphic to a group-measure-space von Neumann algebra.
\section{Preliminary on Herz--Schur multipliers}\label{sec:schur}
Let $G$ be a group. We denote by $\lambda$ the left regular representation
of $G$ on $L^2(G)$, by $C^*_\lambda G$ the reduced group C$^*$-algebra
and by $\cL G$ the group von Neumann algebra of $G$.
The \emph{Fourier algebra} $\cA G$ of $G$ consists of all
functions $\p$ on $G$ such that there are vectors $\xi,\eta \in L^2(G)$
satisfying $\p(x)= \ip{\lambda(x)\xi,\eta} $ for every $x\in G$.
(In other words, $\cA G=L^2(G)*L^2(G)$.)
It is a Banach algebra with the norm $\|\p\|=\inf\{\|\xi\|\|\eta\|\}$, where
the infimum is taken over all $\xi,\eta\in L^2(G)$ as above.
The Fourier algebra $\cA G$ is naturally identified with the predual of
$\cL G$ under the duality pairing
$\ip{\p,\lambda(f)}=\int_G \p f$ for $\p\in \cA G$ and $\lambda(f)\in\cL G$.
If $H$ is a closed subgroup of $G$, then $\p|_H\in \cA H$ for every $\p\in \cA G$.
A continuous function $\p$ on $G$ is called
a \emph{Herz--Schur multiplier} if there are a Hilbert space $\cH$
and bounded continuous functions $\xi,\eta\colon G\to\cH$
such that
$\p(y^{-1}x)=\ip{\xi(x),\eta(y)}$ for every $x,y\in G$.
The Herz--Schur norm of $\p$ is defined by
\[
\|\p\|_{\cb} = \inf\{ \|\xi\|_\infty\|\eta\|_\infty \},
\]
where the infimum is taken over all $\xi,\eta\in C(G,\cH)$ as above.
The Banach space of Herz--Schur multipliers is denoted by $B_2(G)$.
Clearly, one has a contractive embedding of $\cA G$ into $B_2(G)$.
The Herz--Schur norm $\|\p\|_{\cb}$ coincides with the cb-norm
of the corresponding multipliers on $\cL G$ or on $C^*_\lambda G$:
\[
\|\p\|_{\cb}
=\| m_\p\colon \cL G\ni \lambda(f)\mapsto \lambda(\p f)\in \cL G\|_{\cb}
=\| m_\p|_{C^*_\lambda G}\|_{\cb}.
\]
Indeed, $\|\p\|_{\cb}\geq \| m_\p\|_{\cb}$ is easy to see:
Given a factorization $\p(x^{-1}y)=\ip{\xi(x),\eta(y)}$ with $\xi,\eta\in C(G,\cH)$,
we define
$V_\xi\colon L^2(G)\to L^2(G,\cH)$ by
$(V_\xi f)(x)=f(x)\xi(x^{-1})$, and likewise for $V_\eta$.
Then,
$\lambda(\p f)=V_\eta^*(\lambda(f)\otimes1_{\cH})V_\xi$
and $\| m_\p\|_{\cb}\le \|\xi\|_\infty\|\eta\|_\infty$.
We will give a proof of the converse inequality in Lemma~\ref{lem:lmod}, but
sketch it here in the case of amenable groups.
Let $N$ be an amenable group and $\p\in B_2(N)$.
Since the unit character $\tau_0$ is continuous on $C^*_\lambda N$,
the linear functional $\omega_\p=\tau_0\circ m_\p$ is bounded on $C^*_\lambda N$
and satisfies $\|\omega_\p\|\le\|m_\p\|_{\cb}$.
Let $(\pi,\cH)$ be the GNS representation for $|\omega_\p|$ and view $\pi$ as
a continuous unitary $N$-representation.
Then, there are vectors $\xi,\eta\in\cH$
such that $\|\xi\|\|\eta\|=\|\omega_\p\|$ and $\p(x)=\ip{\pi(x)\xi,\eta}$ for every $x\in N$.
(Hence, $\|\omega_\p\|=\|\p\|_{\cb}$.)

\begin{defn}
Let $G$ be a group.
By an \emph{approximate identity} on $G$, we mean a net $(\p_n)$ in $\cA G$
which converges to $1$ uniformly on compacta.
It is \emph{completely bounded} if
\[
\|(\p_n)\|_{\cb} := \sup_n \|\p_n\|_{\cb}<+\infty.
\]
A group $G$ is said to be \emph{weakly amenable} if there is
a completely bounded approximate identity on $G$.
The Cowling--Haagerup constant $\Lambda_{\cb}(G)$ is defined to be
\[
\Lambda_{\cb}(G)=\inf\{ \|(\p_n)\|_{\cb} : \mbox{$(\p_n)$ a c.b.a.i.\ on $G$}\}.
\]
Note that the above infimum is attained. See \cite{ch,bo} for more information.
\end{defn}

It is easy to see that if $H\le G$ is a closed subgroup, then
$\Lambda_{\cb}(H) \le \Lambda_{\cb}(G)$. On this occasion,
we record that the same inequality holds also for a ``random''
or ``ME'' subgroup in the sense of \cite{monod,sako} (cf.\ \cite{cz}).
For this, we only consider countable discrete groups $\Lambda$ and $\G$.
Recall that $\Lambda$ is an ME subgroup of $\G$ if
there is a standard measure space $\Omega$ on which $\Lambda\times\G$ acts
by measure-preserving transformations in such a way that
each of $\Lambda$- and $\G$-actions admits a fundamental domain and
the measure of $\Omega_\G:=\Omega/\G$ is finite.
The action $\Lambda\acts\Omega$ gives rise to a measure-preserving action
$\Lambda\acts\Omega_\G$ and a measurable cocycle
$\alpha\colon\Lambda\times\Omega_\G\to\G$ such that
the action $\Lambda\acts\Omega$ is isomorphic (up to null sets)
to the twisted action $\Lambda\acts\Omega_\G\times\G$, given by
$a(t,g)=(at,\alpha(a,t)g)$ for $a\in\Lambda$, $t\in\Omega_\G$ and $g\in\G$.
The map $\alpha$ satisfies the cocycle identity:
$\alpha(ab,t)=\alpha(a,bt)\alpha(b,t)$ for every $a,b\in\Lambda$
and a.e.\ $t\in\Omega_\G$.
For $\p\in B_2(\G)$, we denote the ``induced'' function on $\Lambda$ by $\p_\alpha$:
\[
\p_\alpha(a)=\int_{\Omega_\G}\p(\alpha(a,t))\,dt.
\]
Here, we normalized the measure so that $|\Omega_\G|=1$.
Since
\[
\p_\alpha(b^{-1}a)=\int_{\Omega_\G}\p(\alpha(b,b^{-1}at)^{-1}\alpha(a,t))\,dt
=\int_{\Omega_\G}\p(\alpha(b,b^{-1}t)^{-1}\alpha(a,a^{-1}t))\,dt,
\]
one has $\p_\alpha\in B_2(\Lambda)$ and $\|\p_\alpha\|_{\cb}\le\|\p\|_{\cb}$.
Suppose now that $\p\in\cA\G$. Then, $\p_\alpha$ is a coefficient of
the unitary $\Lambda$-representation $\sigma$ on $L^2(\Omega)$
induced by the measure-preserving action $\Lambda\acts\Omega$, i.e.,
there are $\xi,\eta\in L^2(\Omega)$ such that
$\p_\alpha(a)=\ip{\sigma(a)\xi,\eta}$.
Since $\Omega$ admits a $\Lambda$-fundamental domain, $\sigma$ is a multiple of the regular representation
and $\p_\alpha\in\cA\Lambda$.
By inducing an approximate identity on $\G$, one sees
that if $\G$ is weakly amenable, then so is $\Lambda$ and
$\Lambda_{\cb}(\Lambda)\le\Lambda_{\cb}(\G)$.
\section{Proof of Theorem~\ref{thmA}}
\begin{lem}\label{lem:lmod}
Let $N$ be an amenable closed normal subgroup of $G$ and $\p\in B_2(G)$.
Then, there are a Hilbert space $\cH$, functions $\xi,\eta\in C(G,\cH)$
and a continuous unitary representation $\pi$ of $N$ on $\cH$ such that
\begin{itemize}
\item
$\|\xi\|_\infty=\|\eta\|_\infty=\|\p\|_{\cb}^{1/2}$;
\item
$\p(y^{-1}x)=\ip{\xi(x),\eta(y)}$ for every $x,y\in G$;
\item
$\pi(a)\xi(x)=\xi(ax)$ and $\pi(a)\eta(y)=\eta(ay)$ for every $a\in N$ and $x,y\in G$.
\end{itemize}
\end{lem}
\begin{proof}
We follow Jolissaint's simple proof (\cite{jolissaint}) of the inequality $\|\p\|_{\cb}\le\| m_\p\|_{\cb}$.
Since $N$ is amenable, the quotient map $q\colon G\to G/N$ extends to
a $*$-homomorphism $q\colon C^*_\lambda G \to C^*_\lambda(G/N)$
between the reduced group C$^*$-algebras.
Since $q\circ m_\p$ is completely bounded on $C^*_\lambda G$,
a Stinespring type factorization theorem (Theorem B.7 in \cite{bo}) yields
a $*$-representation $\pi\colon C^*_\lambda G\to\IB(\cH)$
and operators $V,W\in\IB(L^2(G/N),\cH)$ such that
$\|V\|=\|W\|\le\| q\circ m_\p \|_{\cb}^{1/2}$ and
$(q\circ m_\p )( X ) = W^*\pi(X)V$ for $X\in C^*_\lambda G$.
We view $\pi$ as a continuous unitary representation of $G$.
Then, for a fixed unit vector $\zeta\in L^2(G/N)$,
the maps $\xi(x)=\pi(x)V\lambda_{G/N}(q(x^{-1}))\zeta$ and
$\eta(y)=\pi(y)W\lambda_{G/N}(q(y^{-1}))\zeta$ are continuous,
$\|\xi\|_\infty,\|\eta\|_\infty\le\| m_\p\|_{\cb}^{1/2}$ and
$\p(y^{-1}x)=\ip{\xi(x),\eta(y)}$ for every $x,y\in G$.
Moreover, $\pi(a)\xi(x)=\xi(ax)$ for $a\in N$, because $\lambda_{G/N}(a)=1$.
\end{proof}

We denote by $\p^g$ the right translation of a function $\p$ by $g\in G$, i.e.,
$\p^g(x)=\p(xg^{-1})$.
\begin{lem}\label{lem:aml}
Let $N$ be an amenable group, $\p\in B_2(N)$ and $a\in N$.
Then,
\[
\|\frac{1}{2}\bigl(\p+\p^a\bigr)\|_{\cb}^2
+\|\frac{1}{2}\bigl(\p-\p^a\bigr)\|_{\cb}^2
\le \|\p\|_{\cb}^2.
\]
\end{lem}
\begin{proof}
There are a continuous unitary representation $\pi$ of $N$ on a Hilbert space $\cH$ and
vectors $\xi,\eta\in\cH$ such that $\|\xi\|=\|\eta\|=\|\p\|_{\cb}^{1/2}$ and
$\p(x)=\ip{\pi(x)\xi,\eta}$ for every $x\in N$.
Since $\bigl(\p\pm\p^a\bigr)(x)=\ip{\pi(x)(\xi\pm\pi(a^{-1})\xi),\eta}$, one has
\[
\| \p + \p^a \|_{\cb}^2 + \| \p - \p^a \|_{\cb}^2
\le \| \xi + \pi(a^{-1})\xi \|^2\|\eta\|^2 + \| \xi - \pi(a^{-1})\xi \|^2\|\eta\|^2
= 4\|\p\|_{\cb}^2.
\]
\end{proof}

For $\p\in B_2(G)$, we define $\p^*(x):=\overline{\p(x^{-1})}$, and say $\p$ is
\emph{self-adjoint} if $\p^*=\p$. For any $\p\in B_2(G)$, the function
$(\p+\p^*)/2$ is self-adjoint and $\|(\p+\p^*)/2\|_{\cb}\le\|\p\|_{\cb}$.
Thus every approximate identity can be made self-adjoint without increasing norm.
We fix a closed subgroup $N$ of $G$.
A completely bounded approximate identity $(\p_n)$ on $G$
is said to be \emph{$N$-optimal}
if all $\p_n$ are self-adjoint, $\|(\p_n)\|_{\cb}=\Lambda_{\cb}(G)$ and
\[
\|(\p_n|_N)\|_{\cb}
= \inf\{ \|(\psi_n|_N)\|_{\cb} :
(\psi_n) \mbox{ a c.b.a.i.\ such that }\|(\psi_n)\|_{\cb}=\Lambda_{\cb}(G)\}.
\]
Note that an $N$-optimal approximate identity exists (if $G$ is weakly amenable).

\begin{prop}\label{prop:aap}
Let $G$ be an weakly amenable group and $N$ be an amenable closed normal subgroup of $G$.
Let $(\p_n)$ be an $N$-optimal approximate identity on $G$.
Then, for every $g\in G$ and $a\in N$,
\[
\lim_n\| ( \p_n - \p_n\circ\Ad_g )|_N \|_{\cb}=0
\mbox{ and }
\lim_n \| ( \p_n - \p_n^a )|_N \|_{\cb}=0.
\]
\end{prop}
\begin{proof}
We apply Lemma~\ref{lem:lmod} for each $\p_n$ and
find $(\pi_n,\cH_n,\xi_n,\eta_n)$ satisfying
the conditions stated there. In particular,
$\|\xi\|_\infty=\|\eta\|_\infty\le\Lambda_{\cb}(G)^{1/2}$
and
$\p_n(y^{-1}x)=\ip{\xi_n(x),\eta_n(y)}$ for every $x,y\in G$.
Let $g\in G$ be given and consider $\psi_n=(\p_n+\p_n^g)/2$.
Since $(\psi_n)$ is a completely bounded approximate identity, one must have
$\liminf_n\|\psi_n\|_{\cb}\geq\Lambda_{\cb}(G)$.
Meanwhile, since $\p_n$ is self-adjoint,
\[
\psi_n(y^{-1}x)=\frac{1}{4}\bigl(\ip{\xi_n(x)+\xi_n(xg^{-1}),\eta_n(y)}
 + \ip{\eta_n(x)+\eta_n(xg^{-1}),\xi_n(y)}\bigr)
\]
and hence
\[
\|\psi_n\|_{\cb}\le \left\|\frac{1}{\sqrt{2}}\bigl(
 \frac{\xi_n+\xi_n^g}{2},\,\frac{\eta_n+\eta_n^g}{2}\bigr)
 \right\|_{L^\infty(G,\cH\oplus\cH)}
\left\|\frac{1}{\sqrt{2}}\bigl(\eta_n,\,\xi_n\bigr)
 \right\|_{L^\infty(G,\cH\oplus\cH)}
\le\Lambda_{\cb}(G).
\]
It follows that
\[
\lim_n\left\|\frac{1}{\sqrt{2}}\bigl(
\frac{\xi_n+\xi_n^g}{2},\,\frac{\eta_n+\eta_n^g}{2}\bigr)
\right\|_{L^\infty(G,\cH\oplus\cH)}
=\Lambda_{\cb}(G)^{1/2},
\]
which means that there is a net $z_n\in G$ such that
\[
\lim_n\|\frac{\xi_n(z_n)+\xi_n(z_ng^{-1})}{2}\|=\Lambda_{\cb}(G)^{1/2}
\mbox{ and }
\lim_n\|\frac{\eta_n(z_n)+\eta_n(z_ng^{-1})}{2}\|=\Lambda_{\cb}(G)^{1/2}.
\]
By the parallelogram identity, this implies that
\[
\lim_n\|\xi_n(z_n)-\xi_n(z_ng^{-1})\|=0
\mbox{ and }
\lim_n\|\eta_n(z_n)-\eta_n(z_ng^{-1})\|=0.
\]
The unitary $N$-representation $\pi_n'=\pi_n\circ\Ad_{z_n}$ satisfies
$\pi_n'(a)\xi_n(x)=\xi_n(z_naz_n^{-1}x)$,
\[
\p_n(a) =\ip{\pi_n'(a)\xi_n(z_n),\eta_n(z_n)}
\mbox{ and }
(\p_n\circ\Ad_g)(a) =\ip{\pi_n'(a)\xi_n(z_ng^{-1}),\eta_n(z_ng^{-1})}
\]
for $a\in N$.
It follows that $\|(\p_n - \p_n\circ\Ad_g)|_N\|_{\cb}\to0$.
That $\| ( \p_n - \p_n^a )|_N \|_{\cb}\to0$ follows from
$N$-optimality of $(\p_n)$ and Lemma~\ref{lem:aml}.
\end{proof}

\begin{proof}[Proof of Theorem~\ref{thmA}]
Let $(\p_n)$ be an $N$-optimal approximate identity on $G$ and consider
linear functionals $\omega_n=\tau_0\circ m_{\p_n}$ on $C^*_\lambda N$,
where $\tau_0$ is the unit character on $N$ (see Section~\ref{sec:schur}).
Since $\p_n \in \cA G$, the linear functionals $\omega_n$ extend to
ultraweakly-continuous linear functionals on the group von Neumann algebra $\cL N$.
Indeed, they are nothing but $\p_n|_N\in \cA N=(\cL N)_*$.
One has $\|\omega_n\|\le\Lambda_{\cb}(G)$,
$\omega_n(1_{\cL N})=\p_n(1_N)$ and, by Proposition~\ref{prop:aap},
$\|\omega_n-\omega_n\circ\Ad_g\|\to0$ and $\|\omega_n-\omega_n^a\|\to0$
for every $g\in G$ and $a\in N$.
We consider $\zeta_n:=|\omega_n|^{1/2}\in L^2(N)$ and
$\zeta_n' :=\omega_n|\omega_n|^{-1/2}\in L^2(N)$ so that
$\omega_n(X)=\ip{X\zeta_n,\zeta_n'}$ for $X\in\cL N$.
Here the absolute value and the square root are taken in the sense of
the standard representation $\cL N\subset\IB(L^2(N))$.
(In case where $N$ is abelian,
the Fourier transform $L^2(N)\cong L^2(\widehat{N})$ implements
$\cL N\cong L^\infty(\widehat{N})$ and $(\cL N)_*\cong L^1(\widehat{N})$, and
the absolute value and square root are computed as
ordinary functions on the Pontrjagin dual $\widehat{N}$.)
We note that $\p_n(1)\le\|\zeta_n\|_2^2\le\Lambda_{\cb}(G)$.
By continuity of the absolute value (Proposition III.4.10 in \cite{takesaki})
and the Powers--St{\o}rmer inequality,
one has $\|\zeta_n - \Ad_g \zeta_n \|_2 \to 0$ for every $g\in G$.
Moreover, since
\[
\|\zeta_n\|_2\|\zeta_n'\|_2-\|\frac{\zeta_n + \lambda(a^{-1})\zeta_n}{2}\|_2\|\zeta_n'\|_2
\le\|\omega_n\|-\|\frac{\omega_n+\omega_n^a}{2}\| \to 0,
\]
one has $\|\zeta_n - \lambda(a^{-1})\zeta_n\|_2\to0$ for every $a\in N$.
Thus, any limit point of $(\zeta_n^2)$ in $L^\infty(N)^*$ is a
non-zero positive $G\ltimes N$-invariant linear functional on $L^\infty(N)$.
\end{proof}

\begin{cor}
Let $\G$ and $\Lambda$ be discrete groups with $\Lambda$ non-trivial and $\G$ non-amenable. 
Then the wreath product $\Lambda\wr\G$ is not weakly amenable. 
Also, the group $\SL(2,\IZ)\ltimes\IZ^2$ is not weakly amenable. 
\end{cor}
\begin{proof}
The proof is same as that of Corollary 2.12 in \cite{cartan}. 
We note that the stabilizer of a non-neutral element in $\IZ^2$ 
is an abelian (amenable) subgroup of $\SL(2,\IZ)$.
\end{proof}
\section{Proof of Theorem~\ref{thmB}}\label{sec:b}
We first fix notations.
Throughout this section, $M$ is a finite von Neumann algebra
with a distinguished faithful normal tracial state $\tau$, and
$P$ is an amenable von Neumann subalgebra of $M$.
The \emph{normalizer} $\cN(P)$ of $P$ in $M$ is
\[
\cN(P) = \{ u \in \cU(M) : \Ad_u(P)=P \},
\]
where $\cU(M)$ is the group of the unitary elements of $M$ and $\Ad_u(x)=uxu^*$.
The GNS Hilbert space with respect to the trace $\tau$ is denoted by $L^2(M)$
and the vector in $L^2(M)$ associated with $x\in M$ is denoted by $\hat{x}$, i.e.,
$\ip{\hat{x},\hat{y}}=\tau(y^*x)$ for $x,y\in M$.
The complex conjugate $\bar{M}=\{ \bar{a} : a\in M\}$ of $M$ acts on $L^2(M)$
from the right. Thus there is a $*$-representation $\varsigma$
of the algebraic tensor product $M\otimes\bar{M}$ on $L^2(M)$ defined by
$\varsigma(a\otimes\bar{b})\hat{x}=\widehat{axb^*}$ for $a,b,x\in M$.
We also use the bimodule notation $a\hat{x}b^*$ for $\varsigma(a\otimes\bar{b})\hat{x}$.
Since $P$ is amenable, the $*$-homomorphism $\varsigma|_{M\otimes\bar{P}}$
is continuous with respect to the minimal tensor norm.

\begin{defn}
A von Neumann algebra $M$ is said to have
the \emph{weak$^*$ completely bounded approximation property},
or W$^*$CBAP in short, if there is a net of
ultraweakly-continuous finite-rank maps $(\p_n)$ on $M$
such that $\p_n\to\id_{M}$ in the point-ultraweak topology
and $\sup\|\p_n\|_{\cb}<+\infty$.
\end{defn}

Recall that a finite von Neumann algebra $P$
is amenable (a.k.a.\ hyperfinite, injective, AFD, etc.)
if the trace $\tau$ on $P$ extends to a $P$-central state $\omega$ on $\IB(L^2(P))$.
Here, a state $\omega$ is said to be \emph{$P$-central} if
$\omega\circ\Ad_u=\omega$ for every $u\in\cU(P)$, or equivalently
$\omega(ax)=\omega(xa)$ for every $a\in P$ and $x\in \IB(L^2(P))$.

\begin{defn}
Let $P$ be a finite von Neumann algebra and $\cG$
be a group acting on $P$ by trace-preserving $*$-automorphisms.
We denote by $\sigma$ the corresponding unitary representation of $\cG$ on $L^2(P)$.
The action $\cG\acts P$ is said to be \emph{weakly compact}
if there is a state $\omega$ on $\IB(L^2(P))$ such that
$\omega|_P=\tau$ and
$\omega\circ\Ad_u=\omega$ for every $u\in\sigma(\cG)\cup\cU(P)$.
(This forces $P$ to be amenable.)
A von Neumann subalgebra $P$ of a finite von Neumann algebra $M$
is said to be \emph{weakly compact} in $M$ if
the conjugate action by the normalizer $\cN(P)$ is weakly compact.
See \cite{cartan} for more information.
\end{defn}

If $M$ admits a crossed product decomposition $M=P\rtimes\Lambda$ such that 
the ``core'' $P$ is non-atomic and weakly compact in $M$, then $M$ does not 
have property $\mathrm{(T)}$. Indeed, the hypothesis implies that 
$\cL\Lambda$ is co-amenable in $M$ (Proposition 3.2 in \cite{cartan}), i.e., 
the $M$-$M$ module $L^2\ip{M,e_{\cL\Lambda}}$ contains an approximately 
central vector (see Theorem 2.1 in \cite{cartan}). But since 
$L^2\ip{M,e_{\cL\Lambda}} \cong \oplus_{t\in\Lambda}L^2(P)\otimes L^2(P)$
as a $P$-$P$ module, it does not contain a non-zero central vector.
This proves non property $\mathrm{(T)}$ of $M$.

\begin{lem}\label{lem:wc}
Every $P$-central state $\omega$ on $\IB(L^2(P))$ decomposes uniquely
as a sum $\omega=\omega_{\mathrm{n}}+\omega_{\mathrm{s}}$ of $P$-central
positive linear functionals such that $\omega_{\mathrm{n}}|_P$ is
normal and $\omega_{\mathrm{s}}|_P$ is singular.
A trace-preserving action $\cG\acts P$ is weakly compact if
there is a positive linear functional $\omega$ on $\IB(L^2(P))$ such that
\begin{itemize}
\item $\omega(p)>0$ for every non-zero central projection $p$ in $P$,
\item $\omega\circ\Ad_u=\omega$ for every $u\in\sigma(\cG)\cup\cU(P)$.
\end{itemize}
\end{lem}
\begin{proof}
We denote by $Z$ the center of $P$.
Recall that every tracial state $\tau'$ on $P$ satisfies $\tau'=\tau'|_{Z}\circ E_{Z}$,
where $E_{Z}\colon P\to Z$ is the center-valued trace.
In particular, $\tau'$ is normal on $P$ if and only if it is normal on $Z$.
Let $\omega$ be a $P$-central state and consider the normal/singular decomposition
of the state $\omega|_Z$ (see Definition III.2.15 in \cite{takesaki}).
There is an increasing sequence $(p_n)$ of projections in $Z$ such that
$p_n\nearrow 1$ and $(\omega|_{Z})_{\mathrm{s}}(p_n)=0$
for all $n$ (see Theorem III.3.8 in \cite{takesaki}).
We fix an ultralimit $\Lim$ on $\IN$ and let
$\omega_{\mathrm{n}}(x)=\Lim \omega(p_nx)$ and
$\omega_{\mathrm{s}}=\omega - \omega_{\mathrm{n}}$.
Since $\omega$ is $P$-central, these are $P$-central
positive linear functionals on $\IB(L^2(P))$,
and $\omega|_Z=\omega_{\mathrm{n}}|_Z+\omega_{\mathrm{s}}|_Z$ is the
normal/singular decomposition of $\omega|_Z$.
Suppose that $\omega=\omega_{\mathrm{n}}'+\omega_{\mathrm{s}}'$ is another such decomposition.
Then, since $\omega_{\mathrm{s}}+\omega_{\mathrm{s}}'$ is singular on $Z$, there
is an increasing sequence $(q_n)$ of projections in $Z$ such that
$q_n\nearrow 1$ and $(\omega_{\mathrm{s}}+\omega_{\mathrm{s}}')(q_n)=0$
for all $n$.
It follows that
$\omega_{\mathrm{n}}'(x)=\lim \omega(q_nx)= \omega_{\mathrm{n}}(x)$
for every $x\in\IB(L^2(P))$. This proves the first half of this lemma.
For the second half, we first observe that we may assume $\omega$ is normal on $P$
by uniqueness of the normal/singular decomposition.
Thus, there is $h\in L^1(Z)_+$ such that $\omega(z)=\tau(hz)$ for $z\in Z$.
By assumption, $h$ has full support and is $\cG$-invariant.
Thus $\tilde{\omega}(x):=\Lim\omega((h+n^{-1})^{-1}x)$ defines
a $\cG$-invariant $P$-central state on $\IB(L^2(P))$ such that $\tilde{\tau}|_Z=\tau|_Z$.
\end{proof}

\begin{lem}
Let $\p$ be a completely bounded map on $M$.
Then, there are a $*$-representation of the minimal tensor product
$M\otimes_{\min}\bar{P}$ on a Hilbert space $\cH$ and
operators $V,W\in\IB(L^2(M),\cH)$ such that $\|V\|=\|W\|\le\|\p\|_{\cb}^{1/2}$ and
\[
\tau(y^*\p(a)xb^*)=\ip{\p(a)\hat{x}b^*,\hat{y}}=\ip{\pi(a\otimes\bar{b})V\hat{x},W\hat{y}}
\]
for every $a,x,y\in M$ and $b\in P$.
\end{lem}
\begin{proof}
Since the $*$-representation $\varsigma\colon M\otimes_{\min}\bar{P}\to\IB(L^2(M))$
is continuous, a Stinespring type factorization theorem (Theorem B.7 in \cite{bo}),
applied to the completely bounded map $\varsigma\circ(\p\otimes\id_{\bar{P}})$, yields
a $*$-representation $\pi\colon M\otimes_{\min}\bar{P}\to\IB(\cH)$ and
operators $V,W\in\IB(L^2(M),\cH)$ such that $\|V\|\|W\|\le\|\p\|_{\cb}$ and
\[
\p(a)\hat{x}b^*=\varsigma\bigl((\p\otimes\id_{\bar{P}})(a\otimes\bar{b})\bigr)\hat{x}
=W^*\pi(a\otimes\bar{b})V\hat{x}
\]
for $a,x\in M$ and $b\in P$.
\end{proof}

Since W$^*$CBAP passes to a subalgebra (which is the range of a conditional expectation),
we assume from now on that $P$ is \emph{regular} in $M$, i.e.,
$\cN(P)$ generates $M$ as a von Neumann algebra.
We say a linear map $\p$ on $M$ is \emph{$P$-cb} if
there are a $*$-representation $\pi$ of $M\otimes_{\min}\bar{P}$ on a Hilbert space
$\cH$ and functions $V,W\in\ell_\infty(\cN(P),\cH)$ such that
\[\tag{$\ast$}\label{tag1}
\ip{\p(a)\hat{x}b^*,\hat{y}}=\ip{\pi(a\otimes\bar{b})V(x),W(y)}
\]
for every $a\in M$, $x,y\in\cN(P)$ and $b\in P$.
The $P$-cb norm of $\p$ is defined as
\[
\|\p\|_P=\inf\{ \|V\|_\infty\|W\|_\infty : (\pi,\cH,V,W) \mbox{ satisfies (\ref{tag1})}\}.
\]
It is indeed a norm and the infimum is attained (for the latter fact, use ultraproduct).
By the above lemma, $\|\p\|_P\le\|\p\|_{\cb}$.
By an \emph{approximate identity}, we mean a net $(\p_n)$ of
ultraweakly-continuous finite-rank maps such that $\p_n\to\id_M$
in the point-ultraweak topology and $\sup\|\p_n\|_{P}<+\infty$.
It exists if $M$ has the W$^*$CBAP. We define
\[
\Lambda_{P}(M)=\inf\{ \sup_n\|\p_n\|_{P} : (\p_n) \mbox{ an approximate identity}\}.
\]
For a map $\p$ on $M$, we define $\p^*(a)=\p(a^*)^*$ and say $\p$ is \emph{self-adjoint} if
$\p=\p^*$. We note that if $(\pi,\cH,V,W)$ satisfies (\ref{tag1}) for $\p$, then
$(\pi,\cH,W,V)$ satisfies (\ref{tag1}) for $\p^*$.
In particular, $(\p+\p^*)/2$ is self-adjoint and $\|(\p+\p^*)/2\|_P\le\|\p\|_P$.
Thus, any approximate identity can be made self-adjoint without increasing norm.
For a $P$-cb map $\p$, we define a bounded linear functional $\mu_\p$
on $M\otimes_{\min}\bar{P}$ by
\[
\mu_\p(a\otimes\bar{b}):=\tau(\p(a)b^*)=\ip{\p(a)\hat{1}b^*,\hat{1}}=\ip{\pi(a\otimes\bar{b})V(1),W(1)}.
\]
Note that $\|\mu_\p\|\le\|\p\|_P$. If $\p$ is ultraweakly-continuous and finite-rank,
then $\mu_\p$ extend to an ultraweakly-continuous linear functional on the von Neumann
algebra $M\vt \bar{P}$.

\begin{prop}\label{prop:b}
Let $M$ be a finite von Neumann algebra having the W$^*$CBAP and
$(\p_n)$ be a self-adjoint approximate identity such that $\sup_n\|\p_n\|_{P}=\Lambda_{P}(M)$.
Then, the net $\mu_n:=\mu_{\p_n}|_{P\vt\bar{P}}$ satisfies the following properties:
\begin{itemize}
\item
$\mu_n$ are self-adjoint and ultraweakly-continuous for all $n$;
\item
$\sup\|\mu_n\|\le\Lambda_P(M)$ and $\mu_n(a\otimes\bar{1})\to\tau(a)$ for every $a\in P$;
\item
$\|\mu_n - \mu_n^{v \otimes \bar{v}}\|\to0$ for every $v\in\cU(P)$,
 where $\mu_n^{v \otimes \bar{v}}(a\otimes\bar{b})=\mu_n((a\otimes\bar{b})(v\otimes\bar{v})^*)$;
\item
$\|\mu_n-\mu_n\circ\Ad_{u\otimes\bar{u}}\|\to0$ for every $u\in\cN(P)$.
\end{itemize}
\end{prop}
\begin{proof}
The first two conditions are easy to see.
Let $u\in\cN(P)$ be given, and define $\p_n^u$ by $\p_n^u(a)=\p_n(au^*)u$ for $a\in M$.
We note that $\mu_{\p_n^u}|_{P\vt\bar{P}}=\mu_n^{u \otimes \bar{u}}$ if $u\in\cU(P)$.
Thus, it suffices to show
\[
\lim_n\|\mu_{\p_n}-\mu_{\p_n^u}\|=0
\mbox{ and }
\lim_n\|\mu_{\p_n}-\mu_{\p_n}\circ\Ad_{u\otimes\bar{u}}\|=0.
\]
Take $(\pi_n,\cH_n,V_n,W_n)$ satisfying (\ref{tag1}) and
$\lim\|V_n\|_\infty=\lim\|W_n\|_\infty=\Lambda_P(M)^{1/2}$.
It follows that
\[
\ip{\p_n^u(a) \hat{x} b^*,\hat{y}} = \ip{\p_n(au^*) \widehat{ux} b^*,\hat{y}}
 = \ip{\pi_n(a\otimes\bar{b}) \pi_n(u^*\otimes\bar{1}) V_n(ux), W_n(y)}
\]
for every $a\in M$, $b\in P$ and $x,y\in\cN(P)$.
Hence with $V_n^u(x) = \pi_n(u^*\otimes \bar{1})V_n(ux)$,
the quadruplet $(\pi_n,\cH_n,V_n^u,W_n)$ satisfies (\ref{tag1}) for $\p_n^u$.
Note that $\| V_n^u \|_\infty=\| V_n \|_\infty$. We define $W_n^u$ similarly.
Since $\p_n$ is self-adjoint, $(\pi_n,\cH_n,W_n,V_n)$ (resp.\
$(\pi_n,\cH_n,W_n^u,V_n)$) satisfies
(\ref{tag1}) for $\p_n$ (resp.\ $\p_n^u$), too.
Thus, for $\psi_n=(\p_n+\p_n^u)/2$, one has
\[
\|\psi_n\|_{P}\le
 \left\|\frac{1}{\sqrt{2}}\bigl(\frac{V_n+V_n^u}{2},\frac{W_n+W_n^u}{2}\bigr)
 \right\|_{\ell_\infty(\cN(P),\cH\oplus\cH)}
 \left\|\frac{1}{\sqrt{2}}\bigl(W_n,V_n\bigr)\right\|_{\ell_\infty(\cN(P),\cH\oplus\cH)}.
\]
Meanwhile, since $(\psi_n)$ is an approximate identity,
one must have $\liminf\|\psi_n\|_{P}\geq\Lambda_{P}(M)$.
It follows that
\[
\lim_n\left\|\frac{1}{\sqrt{2}}\bigl(\frac{V_n+V_n^u}{2},\frac{W_n+W_n^u}{2}\bigr)
 \right\|_{\ell_\infty(\cN(P),\cH\oplus\cH)}=\Lambda_{P}(M)^{1/2}
\]
and hence there is a net $(z_n)$ in $\cN(P)$ such that
\[
\lim_n\left\|\frac{1}{\sqrt{2}}\bigl(\frac{(V_n+V_n^u)(z_n)}{2},\frac{(W_n+W_n^u)(z_n)}{2}\bigr)
 \right\|_{\cH\oplus\cH}=\Lambda_{P}(M)^{1/2}.
\]
By the parallelogram identity, this implies that
\[
\lim_n \| V_n(z_n) - V_n^u(z_n) \|=0
\mbox{ and }
\lim_n \| W_n(z_n) - W_n^u(z_n) \|=0.
\]
Let $\pi_n'=\pi_n\circ(\id_M\otimes\Ad_{\bar{z}_n^{-1}})$. Since
\begin{align*}
\mu_{\p_n}(a\otimes \bar{b})
 & = \ip{ \p_n(a) \hat{z}_n \Ad_{z_n^{-1}}(b)^*, \hat{z}_n}
 = \ip{\pi_n'(a\otimes\bar{b})V_n(z_n),W_n(z_n)},\\
\mu_{\p_n^u}(a\otimes \bar{b})
 &= \ip{ \p_n(au^*) \widehat{uz_n} \Ad_{z_n^{-1}}(b)^*, \hat{z}_n}
 = \ip{\pi_n'(a\otimes\bar{b})V_n^u(z_n),W_n(z_n)},
\intertext{ and }
\bigl(\mu_{\p_n}\circ\Ad_{u\otimes\bar{u}}\bigr)(a\otimes \bar{b})
 &= \ip{ \p_n(uau^*) \widehat{uz_n} \Ad_{z_n^{-1}}(b)^*, \widehat{uz_n}}
 = \ip{\pi_n'(a\otimes\bar{b})V_n^u(z_n),W_n^u(z_n)},
\end{align*}
we conclude that $\|\mu_{\p_n}-\mu_{\p_n^u}\|\to0$ and $\|\mu_{\p_n}-\mu_{\p_n}\circ\Ad_{u\otimes\bar{u}}\|\to0$.
\end{proof}

\begin{proof}[Proof of Theorem~\ref{thmB}]
Since $M$ has the W$^*$CBAP, there is a net $(\mu_n)$ satisfying the conclusion of
Proposition~\ref{prop:b}.
We view $\mu_n$ as an element in $L^1(P\vt\bar{P})$ (see Section~2 in \cite{cartan})
and let $\zeta_n=|\mu_n|^{1/2}\in L^2(P\vt\bar{P})$ and $\zeta_n'=\mu_n|\mu_n|^{-1/2}\in L^2(P\vt\bar{P})$
so that $\mu_n(X)=\ip{X\zeta_n,\zeta_n'}$ for $X\in P\vt\bar{P}$.
By continuity of the absolute value (Proposition III.4.10 in \cite{takesaki})
and the Powers--St{\o}rmer inequality,
one has $\| \zeta_n - \Ad_{u\otimes\bar{u}} \zeta_n \|_2\to0$ for every $u\in\cN(P)$.
Since
\[
2\|\mu_n\|\approx\|\mu_n+\mu_n^{v\otimes\bar{v}}\| \le \|\zeta_n+(v\otimes\bar{v})\zeta_n\|_2\|\zeta_n'\|_2
 \le 2\|\zeta_n\|_2\|\zeta_n'\|_2=2\|\mu_n\|,
\]
one also has $\|\zeta_n-(v\otimes\bar{v})\zeta_n\|\to0$ for every $v\in\cU(P)$.
Now, fix an ultralimit $\Lim$ and define $\omega$ on $\IB(L^2(P))$ by
$\omega(x)=\Lim\ip{(x\otimes\bar{1})\zeta_n,\zeta_n}$.
Then $\omega$ is an $\cN(P)$-invariant $P$-central positive linear functional satisfying
\[
\omega(p)=\Lim_n|\mu_n|(p\otimes\bar{1})\geq\Lim_n|\mu_n(p\otimes\bar{1})|=\tau(p)
\]
for every central projection $p$ in $P$. By Lemma~\ref{lem:wc}, we are done.
\end{proof}

\end{document}